\documentclass[a4paper,reqno, 11pt]{amsart}

\usepackage{amsmath}
\usepackage{paralist}
\usepackage{graphics}
\usepackage{epsfig}
\usepackage{amsfonts}
\usepackage{amssymb}
\usepackage{color}
\usepackage{epstopdf}

\newtheorem{theorem}{Theorem}[section]
\newtheorem{lemma}[theorem]{Lemma}

\def\R{{\Bbb R}}

\def\ifl{\iffalse }

\def\bc{\begin{center}}       \def\ec{\end{center}}
\def\ba{\begin{array}}        \def\ea{\end{array}}
\def\be{\begin{equation}}     \def\ee{\end{equation}}
\def\bea{\begin{eqnarray}}    \def\eea{\end{eqnarray}}
\def\beaa{\begin{eqnarray*}}  \def\eeaa{\end{eqnarray*}}

\numberwithin{equation}{section}

\newtheorem{remark}[theorem]{Remark}
\numberwithin{equation}{section}

\begin{document}
\author{Hai-Yang Jin}
\address{School of Mathematics, South China University of Technology, Guangzhou 510640, China}
\email{mahyjin@scut.edu.cn}
\author{Tian Xiang}
\address{Institute for Mathematical Sciences, Renmin University of China, Bejing, 100872, China}
\email{txiang@ruc.edu.cn}

\title[Convergence rate for the chemotaxis-Navier-Stokes system with competitive kinetics ]{Convergence rates of solutions for a two-species chemotaxis-Navier-Stokes sytstem with competitive kinetics}

\begin{abstract} In this paper, we study the convergence rates of solutions to the two-species chemotaxis-Navier-Stokes system with Lotka-Volterra competitive kinetics:
 \begin{equation*}
     \begin{cases}
         (n_1)_t + u\cdot\nabla n_1 =
          \Delta n_1 - \chi_1\nabla\cdot(n_1\nabla c)
          + \mu_1n_1(1- n_1 - a_1n_2),
         &x\in \Omega,\ t>0,
 \\
         (n_2)_t + u\cdot\nabla n_2 =
          \Delta n_2 - \chi_2\nabla\cdot(n_2\nabla c)
          + \mu_2n_2(1- a_2n_1 - n_2),
         &x\in \Omega,\ t>0,
 \\
       \ \ \ \ \  \   c_t + u\cdot\nabla c = \Delta c -(\alpha n_1 + \beta n_2)c,
          &x \in \Omega,\ t>0,
 \\
       \  u_t  + \kappa (u\cdot\nabla) u
          = \Delta u + \nabla P
            + (\gamma n_1 + \delta n_2)\nabla\phi,
            \quad \nabla\cdot u = 0,
          &x \in \Omega,\ t>0
     \end{cases}
 \end{equation*}
 under homogeneous Neumann boundary conditions for $n_1,n_2,c$ and no-slip boundary condition for $u$ in a bounded domain $\Omega \subset \R^d(d\in\{2,3\})$ with smooth boundary. The global existence, boundedness and  stabilization of solutions  have been  obtained in $2$-D \cite {HKMY_1} and $3$-D for $\kappa=0$ and $\frac{\mu_i}{\chi_i} (i=1, 2)$ being sufficiently large \cite{Xinru_et al17}. Here, we examine convergence and derive the {\it explicit rates}  of convergence for any supposedly given global bounded classical solution; more specifically, we show that

 \begin{itemize}
 \item[$\bullet$] when  $a_1, a_2 \in (0, 1)$,  the  global-in-time  bounded  classical solution components  $(n_1,n_2, u)$  converge at least exponentially  to  $(\frac{1 - a_1}{1 - a_1a_2},\frac{1 - a_2}{1 - a_1a_2},0)$ as $t\to \infty$;
\item[$\bullet$] when  $a_1\geq 1>a_2$,  the  global-in-time  and bounded  classical solution components $(n_1,n_2, u)$  converge at least algebraically to $(0,1,0)$ as $t\to \infty$;
\item[$\bullet$] when  $a_2\geq 1>a_1$,  the   global-in-time and  bounded  classical  solution components $(n_1,n_2, u)$  converge at least algebraically to $(1,0,0)$ as $t\to \infty$;
\item[$\bullet$] in either one of the three cases above, the   global-in-time and  bounded  classical classical solution component  $c$  converges at least exponentially  to $0$ as $t\to \infty$.
 \end{itemize}
Moreover, it is shown that the rate of convergence for $u$ in the first case is expressed in terms of the model parameters and the first eigenvalue of $-\Delta$ in $\Omega$ under homogeneous Dirichlet boundary conditions, and all other rates of  convergence are explicitly expressed only in terms of the model parameters $a_i, \mu_i, \alpha$ and $\beta$ and the space dimension $d$.
\end{abstract}



\maketitle

\numberwithin{equation}{section}
\section{Introduction}
We consider the following two-species chemotaxis-fluid system with competitive terms:
 \begin{equation}\label{P}
     \begin{cases}
         (n_1)_t + u\cdot\nabla n_1 =
          \Delta n_1 - \chi_1\nabla\cdot(n_1\nabla c)
          + \mu_1n_1(1- n_1 - a_1n_2),
         &x\in \Omega,\ t>0,
 \\
         (n_2)_t + u\cdot\nabla n_2 =
          \Delta n_2 - \chi_2\nabla\cdot(n_2\nabla c)
          + \mu_2n_2(1- a_2n_1 - n_2),
         &x\in \Omega,\ t>0,
 \\
       \ \ \ \ \ \    c_t + u\cdot\nabla c = \Delta c -(\alpha n_1 + \beta n_2)c,
          &x \in \Omega,\ t>0,
 \\
       \  u_t  + \kappa (u\cdot\nabla) u
          = \Delta u + \nabla P
            + (\gamma n_1 + \delta n_2)\nabla\phi,
            \quad \nabla\cdot u = 0,
          &x \in \Omega,\ t>0,
 \\
        \partial_\nu n_1
        = \partial_\nu n_2 = \partial_\nu c = 0, \quad
        u = 0,
        &x \in \partial\Omega,\ t>0,
 \\
        n_i(x,0)=n_{i,0}(x),\ 
        c(x,0)=c_0(x),\ u(x,0)=u_0(x),
        &x \in \Omega,\ i=1,2,
     \end{cases}
 \end{equation}
 where $\Omega \subset \R^d(d\geq 2)$ is a bounded domain with smooth boundary $\partial\Omega$ and
$\partial_\nu$ denotes differentiation with respect to the
outward normal of $\partial\Omega$;
$\kappa\in\{0,1\}$
$\chi_1, \chi_2, a_1, a_2 \ge 0$ and
$\mu_1, \mu_2, \alpha, \beta, \gamma, \delta > 0$ are
constants;
$n_{1,0}, n_{2,0}, c_0, u_0, \phi$
are known functions satisfying
 \begin{align}\label{condi;ini1}
   &0 < n_{1,0}, n_{2,0}
   \in C(\overline{\Omega}),
 \quad
   0 < c_0 \in W^{1,q}(\Omega),
 \quad
   u_0 \in D(A^{\vartheta}),
  \\\label{condi;ini2}
   &\phi \in C^{1+\eta}(\overline{\Omega})
 \end{align}
for some $q > d$, $\vartheta \in \left(\frac{3}{4}, 1\right)$,
$\eta > 0$ and $A$ is the Stokes operator.

The system \eqref{P},  an extension of the chemotaxis-fluid system introduced by  Tuval et al. \cite{Tuval_et_al},  depicts the evolution  of two competing species which react on a single chemoattractant in a liquid surrounding environment. Here,  $n_1$ and $n_2$  denote densities of species, $c$ means the chemical concentration, and finally, $u$ and $P$ represent the fluid velocity field and its associated  pressure. So, it is the mixed combination of the complex interaction between  chemotaxis, the Lotka-Volterra kinetics and fluid.

 The model \eqref{P} and its variants have been received considerable attention  in $2$- and $3$-dimensional settings.  In one-species context ($n_2\equiv 0$), global existence of weak (and/or eventual smoothness of weak solutions) and classical solutions and asymptotic behavior  have been investigated, e.g.,  in   \cite{W-2012, W-2014, Winkler_2017_Howfar} without logistic source ($\mu_1=0$) and also the convergence
rate has been  explored \cite{Zhang-Li_2015_fluid} and in \cite{Lankeit_2016, Tao-Winkler_2015_mu23, TW-2016} with  logistic source.

 In  two-species context,  related studies first begin with  fluid-free systems with signal production (in which the asymptotic stability usually depends on some smallness condition on the chemo-sensitivities) to understand the influence of chemotaxis and the Lotka-Volterra kinetics \cite{B-W, Black-Lankeit-Mizukami, MY-2016, Mizukami, N-T_SIAM, N-T_JDE,stinner_tello_winkler}. For the two-species chemotaxis-fluid system with competitive terms \eqref{P}, the global existence, boundedness of classical solutions and stabilization to equilibria  were very recently studied by Hirata et al. \cite{HKMY_1} in the 2-D setting and by Cao et al. \cite{Xinru_et al17} in the $3$-D setting for $\kappa=0$, as precisely stated as follows:
\begin{itemize}
\item[(B2)]  \textbf{(Boundedness in $2$-D \cite{HKMY_1})} In the case that  $\Omega \subset \mathbb{R}^2$ is a
 bounded domain with smooth boundary, let $\chi_1, \chi_2, a_1, a_2 \ge 0$,
$\mu_1, \mu_2, \alpha, \beta, \gamma, \delta > 0$ and let \eqref{condi;ini1} and \eqref{condi;ini2} hold.  The the IBVP  \eqref{P} possesses a unique classical solution $(n_1, n_2, c, u, P)$, up to  addition of constants to  $P$, such that
\begin{align*}
&n_1, n_2 \in C(\overline{\Omega}\times[0, \infty))
\cap C^{2, 1}(\overline{\Omega}\times(0, \infty)),
\\
&c \in C(\overline{\Omega}\times[0, \infty))
\cap C^{2, 1}(\overline{\Omega}\times(0, \infty))
\cap L^{\infty}_{{\rm loc}}([0, \infty); W^{1, q}(\Omega)),
\\
&u \in C(\overline{\Omega}\times[0, \infty))
\cap C^{2, 1}(\overline{\Omega}\times(0, \infty))
\cap L^\infty_{\rm loc}([0,\infty);D(A^{\vartheta})),
\\
&P \in C^{1, 0}(\overline{\Omega} \times (0, \infty)).
\end{align*}
Moreover, there  exists a constant $C>0$ such that for all $t>0$
\be\label{bdd-sol}
   \|n_1(\cdot, t)\|_{L^{\infty}(\Omega)} + \|n_2(\cdot, t)\|_{L^{\infty}(\Omega)}
   + \|c(\cdot, t)\|_{W^{1, q}(\Omega)} + \|u(\cdot, t)\|_{L^{\infty}(\Omega)}
   \leq C.
 \ee
\item[(B3)]  \textbf{(Boundedness in $3$-D \cite{Xinru_et al17})} In the case that  $\Omega \subset \mathbb{R}^3$ is a bounded domain with smooth boundary, besides the assumptions in (B2), let $\kappa=0$. Then there exists a constant $\xi_0 > 0$ such that whenever $\frac{\max\{\chi_1, \chi_2\}}{\min\{\mu_1, \mu_2\}}<\xi_0$, the statements of (B2) hold.

\item[(UC)]  \textbf{(Uniform Convergence)\cite{HKMY_1,Xinru_et al17}} Let  $(n_1, n_2, c, u, P)$ be the solution of \eqref{P} obtained from (B2) or (B3).  Then it fulfills the following convergence properties\/{\rm :}
   \begin{enumerate}
   \item[{{\rm (i)}}] Assume that $a_1, a_2 \in (0, 1)$. Then
       $$
       n_1(\cdot, t) \to N_1,\quad n_2(\cdot, t) \to N_2,\quad c(\cdot, t) \to 0,\quad
       u(\cdot, t) \to 0 \mbox{ in }\ L^{\infty}(\Omega) \mbox{ as }\ t \to \infty,
       $$
where
     \be\label{N1N2-def}
       N_1 := \frac{1 - a_1}{1 - a_1a_2},\quad  \quad \quad N_2 := \frac{1 - a_2}{1 - a_1a_2}.
      \ee
   \item[{{\rm (ii)}}] Assume that $a_1 \geq 1 > a_2$.
     Then
    $$
      n_1(\cdot, t) \to 0,\quad n_2(\cdot, t) \to 1,\quad c(\cdot, t) \to 0,\quad
       u(\cdot, t) \to 0  \mbox{ in }\ L^{\infty}(\Omega)  \mbox{ as }\ t \to \infty.
   $$
   \end{enumerate}

\end{itemize}

In this paper, we study dynamical properties  for any supposedly global-in-time and  bounded  classical solution to \eqref{P}, with particular focus on the model of convergence as well as their explicit rates of convergence. Before proceeding to our main results, let us observe that, the $n_1$- and $n_2$-equations in \eqref{P} are symmetric. Thus,  we should have a result about stabilization to $(0,1, 0,0)$. This was not mentioned in related works, cf.  \cite{B-W,  Xinru_et al17,HKMY_1,Mizukami}.  Now, let $\lambda_P$ denote  the  Poincar\'{e}  constant, cf. \eqref{poincare},  and, finally,  let
\be\label{kappa-def}
\kappa=\frac{1}{2}\min\left\{\frac{(1-a_1a_2)\mu_1\min\Bigr\{\frac{1}{2}, \frac{a_1}{(1+a_1a_2)a_2}\Bigr\}}{\max\{\frac{1}{N_1},\frac{a_1\mu_1}{a_2\mu_2N_2}\}}, \ \  (\alpha N_1+\beta N_2)\right\}.
\ee
Then we are at the position to state our main results on exponential and algebraic convergence of bounded solutions to \eqref{P}.
\begin{theorem}\label{convergence thm}  Let  $\Omega\subset \mathbb{R}^d(d\in\{2, 3\})$  be a bounded and smooth domain and  let \eqref{condi;ini1} and \eqref{condi;ini2} be in force, and, finally, let  $(n_1, n_2, c, u, P)$ be  a global classical  solution of  \eqref{P} with uniform-in-time bound. Then this solution enjoys  the following decay  properties.
   \begin{itemize}
   \item[(I)] When  $a_1, a_2 \in (0, 1)$, the solution components  $(n_1,n_2, u)$ converge  at least  exponentially  to $(N_1,N_2,0)$ in the following way:
\be\label{exp-convergence-n1-n2-u}
\begin{cases}
\|n_1(\cdot, t)-N_1\|_{L^\infty(\Omega)}\leq m_1e^{-\frac{\kappa}{d+2} t}, & \forall t\geq 0, \\[0.25cm]
\|n_2(\cdot, t)-N_2\|_{L^\infty(\Omega)}\leq m_2e^{-\frac{\kappa}{d+2} t}, & \forall t\geq 0, \\[0.25cm]
\|u(\cdot, t)\|_{L^\infty(\Omega)}\leq m_3e^{-\frac{\epsilon}{d+2}\min\{\lambda_P,\frac{\kappa}{2}\}t}, & \forall t\geq 0.
\end{cases}
\ee
\item[(II)] When $a_1\geq 1>a_2$,  the solution components $(n_1,n_2, u)$  converge at least algebraically to $(0,1,0)$ in the following way:
      \be\label{alg-convergence-n1-n2-u II}
\begin{cases}
\|n_1(\cdot, t)\|_{L^\infty(\Omega)}\leq m_4 (t+1)^{-\frac{1}{d+1}}, & \forall t\geq 0, \\[0.25cm]
\|n_2(\cdot, t)-1\|_{L^\infty(\Omega)}\leq m_5(t+1)^{-\frac{1}{d+2}}, & \forall t\geq 0, \\[0.25cm]
\|u(\cdot, t)\|_{L^\infty(\Omega)}\leq m_6(t+1)^{-\frac{\epsilon}{d+2}}, & \forall t\geq 0.
\end{cases}
\ee
\item[(III)] When $a_2\geq 1>a_1$, the solution components $(n_1,n_2, u)$  converge at least algebraically to $(1,0,0)$ in the following way:
     \be\label{alg-convergence-n1-n2-u III}
\begin{cases}
\|n_1(\cdot, t)-1\|_{L^\infty(\Omega)}\leq m_7 (t+1)^{-\frac{1}{d+2}}, & \forall t\geq 0, \\[0.25cm]
\|n_2(\cdot, t)\|_{L^\infty(\Omega)}\leq m_8(t+1)^{-\frac{1}{d+1}}, & \forall t\geq 0, \\[0.25cm]
\|u(\cdot, t)\|_{L^\infty(\Omega)}\leq m_9(t+1)^{-\frac{\epsilon}{d+2}}, & \forall t\geq 0.
\end{cases}
\ee
\item[(IV)] In either one of the three cases above, the solution component  $c$ converges  at least  exponentially  to $0$  in the following way:
    \be\label{c-exp-rate}
    \|c(\cdot,t)\|_{L^\infty(\Omega)}\leq m_{10} e^{-\frac{(\alpha \hat{N}_1+\beta \hat{N}_2)}{2}t}, \quad \forall t\geq 0,
    \ee
   where  $(\hat{N}_1, \hat{N}_2)=(N_1,N_2)$ in Case (I),  $(\hat{N}_1, \hat{N}_2)=(0,1)$ in Case (II), and $(\hat{N}_1, \hat{N}_2)=(1,0)$ in Case (III).
   \end{itemize}
Here, $\epsilon\in (0,1)$ is arbitrarily given, only $m_3, m_6$ and $m_9$ depend on $\epsilon$;  all $m_i(i=1,2,3,\cdots, 9)$ are suitably large constants depending on  the initial data $n_{1,0}, n_{2,0}, c_0, u_0, \phi$ and Sobolev embedding constants but not on time $t$, see Section 4. Moreover,
$$
m_4\geq O(1)\Bigr(1+(1-a_2)^{-\frac{1}{d+1}}\Bigr),\ \  m_5\geq O(1)\Bigr(1+(1-a_2)^{-\frac{1}{d+2}}\Bigr), \ \ m_6\geq O(1)\Bigr(1+(1-a_2)^{-\frac{\epsilon}{d+2}}\Bigr)
$$
and
$$
m_7\geq O(1)\Bigr(1+(1-a_1)^{-\frac{1}{d+2}}\Bigr),\ \  m_8\geq O(1)\Bigr(1+(1-a_1)^{-\frac{1}{d+1}}\Bigr), \ \ m_9\geq O(1)\Bigr(1+(1-a_1)^{-\frac{\epsilon}{d+2}}\Bigr).
$$
\end{theorem}
From these estimates, we see the facts that $1-a_1a_2>0$, $1-a_2>0$ and $1-a_1>0$ are  very important in Case (I),  (II) and (III), respectively.

\begin{remark}  In our argument, we don't need any restriction on the space dimension $d$. Thus,  Theorem \ref{convergence thm} works equally well in any dimension as long as the solution is global-in-time and bounded. While, by the well-known difficulty about the Navier-Stokes system,   we restrict ourselves to the physically relevant cases $d=2$ and $d=3$.
\end{remark}

With certain regularity and dissipation properties of global bounded solutions, the argument for the proof of convergence is quite known and developed, cf. \cite{B-W, Xinru_et al17,HKMY_1,Mizukami, Tao-Winkler_2015_mu23, TW-2015-SIAMJMA, TW-2016, W-2014} for example.  The strategy for obtaining  the explicit rates of convergence as described in Theorem \ref{convergence thm} consists mainly of four steps. In the first step, we present more strong regularity properties, e.g.,  $W^{1,\infty}$-regularity for $n_i$, $W^{1,p}$-regularity for $u$ with any finite $p$, and $W^{2,\infty}$-regularity for $c$,   for any  bounded solution of \eqref{P} than those shown in  \cite{HKMY_1,Xinru_et al17}; this are done in Section 2. In the crucial second  step done in Section 3, we  use refined computations to make those widely known Lyapunov functionals (cf. eg. \cite{B-W,  Xinru_et al17,HKMY_1,Mizukami})  explicit, which will enable us to derive the explicit rates of convergence. Armed with the information provided by Step two, we then move on to calculate precisely the rates of convergence in $L^1$-and  $L^2$-norm for the considered bounded solution, and related necessary estimates are also studied in great details. These constitute our Step three and are conducted in Section 4. Finally, thanks to the improved regularities, we apply the well-known  Gagliardo-Nirenberg interpolation inequality to pass the obtained  $L^1$- and $L^2$-convergence  to  the  $L^\infty$-convergence; these are our Step four and are also done in Section 4.

\section{Regularities of bounded solutions}
Let  $(n_1, n_2, c, u, P)$ be a supposedly given  global-in-time and bounded  classical solution to \eqref{P} in the sense of \eqref{bdd-sol}. In this section, we provide more strong regularity properties for any such bounded solution than those shown in  \cite{HKMY_1,Xinru_et al17}, which are needed to achieve our desired  rates of convergence  in $L^\infty$-norm. We start with  the regularity of $u$ and $c$.
\begin{lemma}\label{BS}  Let $\Omega\subset \mathbb{R}^d$  be a bounded and smooth domain.  For $d<p<\infty$, there exists a constant $C>0$ such that
\begin{equation}\label{Lu1e-1}
\|u(\cdot,t)\|_{W^{1,p}} \leq C, \quad \forall  t>1
\end{equation}
and
\begin{equation}\label{Ce}
\|c(\cdot,t)\|_{W^{1,\infty}}+\|\Delta c(\cdot,t)\|_{L^\infty}\leq C , \quad \forall  t>1.
\end{equation}
\end{lemma}
\begin{proof} Using the essentially same  argument as in \cite[Lemma 6.3]{W-2014}, we obtain \eqref{Lu1e-1}. The $W^{1,\infty}$-boundedness of $c$ can be seen in \cite[Lemma 3.9]{HKMY_1} and \cite[Lemma 3.12]{TW-2016}. With these and  the $L^\infty$-boundedness of $n_1, n_2$ and $u$, an direct application of the standard parabolic schauder theory (cf. \cite{Ladyzenskaja, PV93} to the third equation in \eqref{P} yields \eqref{Ce}.
\end{proof}

With the regularity properties  in Lemma \ref{BS} at hand, we now utilize the quite commonly used arguments (cf. \cite{TW-2015-SIAMJMA, TW-2016}) to  show the following $W^{1,\infty}$-regularity of $n_1$ and $n_2$.

\begin{lemma}\label{KS} There exists a constant $C>0$ such that
\begin{equation}\label{KS-1}
\|n_1(\cdot,t)\|_{W^{1,\infty}}+\|n_2(\cdot,t)\|_{W^{1,\infty}}\leq C, \quad \forall  t>1.
\end{equation}
\end{lemma}
\begin{proof}
First, we show there exists a $c_1>0$ such that
\be\label{KS-2}
\|n_1(\cdot,t)\|_{W^{1,\infty}}\leq c_1, \quad \forall  t>1.
\ee
To this end, for any $T>2$, we let
\begin{equation*}
M(T):= \sup\limits_{t\in(2,T)}\|\nabla n_1(\cdot,t)\|_{L^\infty}.
\end{equation*}
Since clearly $\nabla n_1$ is continuous on $\overline{\Omega} \times [0,T]$, it follows that  $M(T)$ is finite. Moreover, since by our universal assumption $n_1$ is bounded in $L^\infty(\Omega\times (0,\infty))$, to prove \eqref{KS-2}, it is sufficient to derive the existence of $c_2>0$ satisfying
\begin{equation}\label{KS-3}
M(T)\leq c_2, \quad \forall  T>2.
\end{equation}
To achieve \eqref{KS-3}, for any given $t\in(2,T)$,  using  the variation-of-constants formula to the first equation  in  \eqref{P}, we get
\begin{equation*}
\begin{split}
n_1(\cdot,t)=&e^{\Delta} n_1(\cdot,t-1)- \chi \int_{t-1}^t e^{(t-s)\Delta}\nabla \cdot(n_1(\cdot,s)\nabla c(\cdot,s))ds\\
&\ \ -\int_{t-1}^t e^{(t-s)\Delta}u(\cdot,s)\cdot\nabla n_1(\cdot,s)ds\\
&\ \ +\mu_1 \int_{t-1}^t e^{(t-s)\Delta }n_1(\cdot,s)(1- n_1(\cdot,s) - a_1n_2(\cdot,s))ds,
\end{split}
\end{equation*}
which implies
\begin{equation}\label{KS-3}
\begin{split}
\|\nabla n_1(\cdot,t)\|_{L^\infty}
\leq& \|\nabla e^{\Delta}n_1(\cdot,t-1)\|_{L^\infty}+\chi\int_{t-1}^t\|\nabla e^{(t-s)\Delta}\nabla \cdot(n_1(\cdot,s)\nabla c(\cdot,s))\|_{L^\infty}ds\\
       &+\int_{t-1}^t\|\nabla e^{(t-s)\Delta} u(\cdot,s)\cdot\nabla n_1(\cdot,s)\|_{L^\infty}ds\\
       &+\mu_1\int_{t-1}^t \|\nabla e^{(t-s)\Delta }n_1(\cdot,s)(1- n_1(\cdot,s) - a_1n_2(\cdot,s))\|_{L^\infty}ds\\
 = & I_1+I_2+I_3+I_4.
\end{split}
\end{equation}
Next, we shall employ  the widely  known smoothing $L^p$-$L^q$ properties of the Neumann heat semigroup $\{e^{t\Delta}\}_{t\geq0}$ in $\Omega$ (see \cite{Winkler_2010_Aggreagation, Cao15, Henry_1981} for instance) to estimate $I_i, i=1,2,3,4$.

Thanks to the boundedness of $n_1,n_2,u$ in $\Omega\times (1,\infty)$, \eqref{Lu1e-1} and \eqref{Ce}, we employ those smoothing Neumann heat semigroup estimates to  obtained that
\begin{equation}\label{I1}
I_1= \|\nabla e^{\Delta}n_1(\cdot,t-1)\|_{L^\infty}\leq c_3\|n_1(\cdot,t-1)\|_{L^\infty}\leq c_4
\end{equation}
and that
\begin{equation}\label{I2}
\begin{split}
I_2
&=\chi\int_{t-1}^t\|\nabla e^{(t-s)\Delta}\nabla \cdot(n_1(\cdot,s)\nabla c(\cdot,s))\|_{L^\infty}ds\\
&\leq c_5  \int_{t-1}^t \Bigr[1+(t-s)^{-\frac{1}{2}-\frac{n}{2p}}\Bigr]e^{-\lambda_1(t-s)}\|\nabla \cdot(n_1(\cdot,s)\nabla c(\cdot,s))\|_{L^p}ds\\
&\leq c_5 \int_{t-1}^t \Bigr[1+(t-s)^{-\frac{1}{2}-\frac{n}{2p}}\Bigr]e^{-\lambda_1(t-s)} \|\nabla n_1(\cdot,s)\cdot\nabla c(\cdot,s)\|_{L^p}ds\\
&\ \ \ \ +c_5\int_{t-1}^t \Bigr[1+(t-s)^{-\frac{1}{2}-\frac{n}{2p}}\Bigr]e^{-\lambda_1(t-s)}\|n_1(\cdot,s)\Delta c(\cdot,s)\|_{L^p}ds\\
&\leq c_6 \int_{t-1}^t \Bigr[1+(t-s)^{-\frac{1}{2}-\frac{n}{2p}}\Bigr]e^{-\lambda_1(t-s)}\|\nabla n_1(\cdot,s)\|_{L^p}ds+c_7,
\end{split}
\end{equation}
where $\lambda_1(>0)$ is the first nonzero eigenvalue of $-\Delta$ under homogeneous boundary condition and we have used the choice of $p>n$ to ensure the finiteness of the Gamma integral.  Similarly, we can estimate $I_3$ as follows:
\begin{equation}\label{I3}
\begin{split}
I_3&=\int_{t-1}^t\|\nabla e^{(t-s)\Delta}  u(\cdot,s)\cdot\nabla n_1(\cdot,s)\|_{L^\infty}ds\\
&\leq c_8 \int_{t-1}^t \Bigr[1+(t-s)^{-\frac{1}{2}-\frac{n}{2p}}\Bigr]e^{-\lambda_1(t-s)}\|\nabla n_1(\cdot,s)\|_{L^p}ds+c_9.
\end{split}
\end{equation}
At last, using the boundedness of $n_1$ and $n_2$ again, one has
\begin{equation}\label{I4}
\begin{split}
I_4 =&\mu_1\int_{t-1}^t \|\nabla e^{(t-s)\Delta }n_1(\cdot,s)(1- n_1(\cdot,s) - a_1n_2(\cdot,s))\|_{L^\infty}ds\\
\leq &c_{10}\int_{t-1}^{t}\Bigr[1+(t-s)^{-\frac{1}{2}}\Bigr]e^{-\lambda_1(t-s)}ds\leq c_{10}\int_{0}^{1}(1+\tau^{-\frac{1}{2}})e^{-\lambda_1\tau}ds\leq (\frac{1}{\lambda_1}+2)c_{10}.
\end{split}
\end{equation}
Substituting \eqref{I1}, \eqref{I2}, \eqref{I3} and \eqref{I4} into \eqref{KS-3}, we infer that
\begin{equation}\label{KS-4}
\|\nabla n_1(\cdot,t)\|_{L^\infty}\leq   c_{11} \int_{t-1}^t \Bigr[1+(t-s)^{-\frac{1}{2}-\frac{n}{2p}}\Bigr]e^{-\lambda_1(t-s)}\|\nabla n_1(\cdot,s)\|_{L^p}ds+c_{12}.
\end{equation}
Then invoking the Gagliardo-Nirenberg  inequality, the smoothness and hence boundedness of $\nabla n_1$  on $\overline{\Omega}\times[1,2]$, and the definition of $M(T)$, we estimate
\begin{equation}\label{KS-5}
\begin{split}
\|\nabla n_1(\cdot,s)\|_{L^p}
&\leq c_{13}\|\nabla n_1(\cdot,s)\|_{L^\infty}^\theta \|n_1(\cdot,s)\|_{L^\infty}^{1-\theta}+c_{13}\|n_1(\cdot,s)\|_{L^\infty}\\
&\leq c_{14}(M^\theta(T)+1), \quad \forall  s\in(1,T),
\end{split}
\end{equation}
where  $\theta=\frac{p-n}{p}\in(0,1)$ due to $p>n$.

Finally, since $\frac{1}{2}+\frac{n}{2p}<1$, then  a substitution of  \eqref{KS-5} into \eqref{KS-4} entails
\begin{equation*}
M(T)\leq c_{15} M^{\theta}(T)+c_{16}, \quad \forall  T>2,
\end{equation*}
which upon a use of elementary inequality gives
\begin{equation*}
M(T)\leq \max\{2c_{16},\ \ (2c_{15})^\frac{1}{1-\theta}\}, \quad \forall  T>2,
\end{equation*}
and hence \eqref{KS-2} follows.

The argument  done for $n_1$ can also be similarly applied to  $n_2$ to find  that
$$
\|n_2(\cdot,t)\|_{W^{1,\infty}}\leq c_{17}, \quad \forall  t>1.
$$
This along with \eqref{KS-2} yields simply \eqref{KS-2}, finishing the proof of the lemma.\end{proof}

\section{Existence of explicit Lyapunov functionals}
From boundedness to convergence, besides enough information on regularity, we still need some decaying estimates of bounded solutions under investigation. For the latter, the availability of a Lyapunov functional is crucial, see \cite{B-W,  Xinru_et al17,HKMY_1,Mizukami} for instance.  In this section, for our purpose, we particularize those known Lyapunov functionals used in those papers  to obtain the explicit rates of convergence as stated in Theorem \ref{convergence thm}.  Let us start with the case of $a_1, a_2 \in (0, 1)$. In this case,  the  explicit  Lyapunov functional that we obtain for the chemotaxis-fluid system \eqref{P}  reads as follows:
\begin{lemma}\label{LY1} Define
\be\begin{split}\label{E1-def}
  E_1:=&\int_\Omega
       \left(
         n_1-N_1-N_1\log \frac{n_1}{N_1}
       \right)
       +
      \frac{a_1\mu_1}{a_2\mu_2}\int_\Omega
       \left(
         n_2-N_2-N_2\log \frac{n_2}{N_2}
       \right)\\[0.25cm]
       &+  \frac{1}{2}\left(\frac{N_1\chi_1^2}{4}+\frac{a_1\mu_1N_2\chi_2^2}{4a_2\mu_2}+1\right)
       \int_\Omega c^2
\end{split}
\ee
and
\begin{align*}
F_1:=\int_\Omega (n_1-N_1)^2
     +
     \int_\Omega (n_2-N_2)^2.
\end{align*}
Then, in the case of $a_1, a_2 \in (0, 1)$,  the nonnegative functions $E_1$ and $F_1$ satisfy
\begin{align}\label{EF1}
  \frac d{dt}E_1(t)
  \leq
  - (1-a_1a_2)\mu_1\min\Bigr\{\frac{1}{2}, \frac{a_1}{(1+a_1a_2)a_2}\Bigr\} F_1(t):=-\tau F_1(t), \quad \forall  t>0.
\end{align}
\end{lemma}
\begin{proof} By honest differentiation of $E_1$ in \eqref{E1-def} and using elementary Cauchy-Schwarz inequality, one can easily derive the dissipation estimate \eqref{EF1}; or alternatively, in the proof of  \cite[Lemma 4.1]{HKMY_1}, by taking
$$
k_1=\frac{a_1\mu_1}{a_2\mu_2}, \quad l_1=(\frac{N_1\chi_1^2}{4}+\frac{a_1\mu_1N_2\chi_2^2}{4a_2\mu_2}+1), \quad \epsilon=\frac{\mu_1}{2}(1-a_1a_2),
$$
upon honest calculations, one can easily arrive at \eqref{EF1}.\end{proof}

\ \ For the  purpose of driving our explicit Lyapunov functionals in Cases (II) and (III), we wish to perform   honest computations here. We illustrate it for Case (II), i.e., $a_1\geq 1>a_2$. In this case,   from \eqref{P}, the fact that $a_1\geq 1$ and the positivity of $n_1,n_2$, we calculate  that
$$
\frac{d}{dt}\int_\Omega n_1=\mu_1\int_\Omega(1-n_1-a_1 n_2)n_1\leq -\mu_1\int_\Omega n_1^2-\mu_1 \int_\Omega n_1(n_2-1),
$$
$$
\frac{d}{dt}\int_\Omega (n_2-1-\log n_2)=-\int_\Omega \frac{|\nabla n_2|^2}{n_2^2}+
\chi_2\int_\Omega \frac{\nabla n_2}{n_2}\cdot \nabla c-\mu_2\int_\Omega ( n_2-1)^2-a_2\mu_2\int_\Omega n_1( n_2-1)
$$
as well as
$$
\frac{1}{2}\frac{d}{dt}\int_\Omega c^2=-\int_\Omega  |\nabla c|^2- \int_\Omega (\alpha n_1 + \beta n_2)c^2.
$$
Therefore, for any positive constants $\sigma_1,  \sigma_2$ and $\eta\in(0,1)$, in view of the passivity of $n_i,\alpha$ and $ \beta$,  a clear linear combination of the three estimates above shows
\be\label{LY-re-comp}
\begin{split}
&-\frac{d}{dt}\int_\Omega \left [n_1
       +
       \sigma_1
       \left(
         n_2-1-\log n_2
       \right)
       + \frac{\sigma_2}{2}
         c^2\right]\\
&\geq\int_\Omega\left[\mu_1n_1^2+(\mu_1+a_2\mu_2\sigma_1)n_1(n_2-1)+\mu_2\sigma_1 ( n_2-1)^2\right]\\
&\ \ +\int_\Omega\Bigr[\sigma_1\frac{|\nabla n_2|^2}{n_2^2}-\chi_2\sigma_1 \frac{\nabla n_2}{n_2}\cdot \nabla c+\sigma_2|\nabla c|^2\Bigr]\\
&=\int_\Omega\left\{\Bigr[\sqrt{\mu_2\sigma_1\eta} ( n_2-1)+\frac{(\mu_1+a_2\mu_2\sigma_1)}{2\sqrt{\mu_2\sigma_1\eta}}n_1\Bigr]^2 +\Bigr[\mu_1-
\frac{(\mu_1+a_2\mu_2\sigma_1)^2}{4\mu_2\sigma_1\eta}\Bigr]n_1^2\right\}\\
&\ \ \ + \mu_2\sigma_1(1-\eta)\int_\Omega  ( n_2-1)^2+\int_\Omega\Bigr[\Bigr(\sqrt{\sigma_1}\frac{\nabla n_2}{n_2}-\frac{\chi_2\sqrt{\sigma_1}}{2} \nabla c\Bigr)^2+\Bigr(\sigma_2-\frac{\chi_2^2\sigma_1}{4}\Bigr)|\nabla c|^2\Bigr]\\
&\geq  \mu_2\sigma_1(1-\eta)\int_\Omega  ( n_2-1)^2+[\mu_1-
\frac{(\mu_1+a_2\mu_2\sigma_1)^2}{4\mu_2\sigma_1\eta}]\int_\Omega n_1^2+ (\sigma_2-\frac{\chi_2^2\sigma_1}{4})\int_\Omega |\nabla c|^2.
\end{split}
\ee
 With these calculations above, we obtain the next explicit decay property, which is a specification  of  \cite[Lemma 4.3]{HKMY_1}, see also  \cite[Section 4.2]{Xinru_et al17}.
 \begin{lemma}\label{C2*}
Define
\begin{align*}
  E_2:=\int_\Omega n_1
       +
       \frac{\mu_1}{a_2\mu_2}\int_\Omega
       \left(
         n_2-1-\log n_2
       \right)
       + \frac{\mu_1\chi_2^2}{8a_2\mu_2}
       \int_\Omega c^2
\end{align*}
and
\begin{align*}
F_2:=\int_\Omega n_1^2
     +
     \int_\Omega (n_2-1)^2.
\end{align*}
Then, in the case of $a_1\geq 1>a_2$,  the nonnegative functions $E_2$ and $F_2$ satisfy
\begin{align}\label{C2*-1}
  \frac d{dt}E_2(t)
  \leq
  - (1-a_2)\mu_1\min\Bigr\{\frac{1}{2a_2}, \frac{1}{1+a_2}\Bigr\} F_2(t):=-\sigma F_2(t), \quad \forall  t>0.
\end{align}
\end{lemma}
\begin{proof} The  fact that $a_2<1$ allows us to select
$$
\sigma_1= \frac{\mu_1}{a_2\mu_2}, \quad \sigma_2=\frac{\mu_1\chi_2^2}{4a_2\mu_2}, \quad \quad \eta=\frac{1+a_2}{2}\in (0, 1),
$$
then, upon a plain calculation from \eqref{LY-re-comp},  we obtain the dissipation inequality \eqref{C2*-1}.
\end{proof}
By the symmetry of the $n_1$-and $n_2$-equations in \eqref{P}, when $a_2\geq 1>a_1$, using similar arguments leading to Lemma \ref{C2*}, we have a dissipation inequality as follows:
\begin{lemma}\label{C3*}
Define
\begin{align*}
  E_3:=\int_\Omega \left(
         n_1-1-\log n_1
       \right)
       +
       \frac{a_1\mu_1}{\mu_2}\int_\Omega n_2+ \frac{\chi_1^2}{8}\int_\Omega c^2
\end{align*}
and
\begin{align*}
F_3:=\int_\Omega (n_1-1)^2
     +
     \int_\Omega n_2^2.
\end{align*}
Then, in the case of $a_2\geq 1>a_1$,  the nonnegative functions $E_3$ and $F_3$ satisfy
\be\label{decay-lyIII}
  \frac d{dt}E_3(t)
  \leq
  - (1-a_1)\mu_1\min\Bigr\{\frac{1}{2}, \frac{a_1}{1+a_1}\Bigr\} F_3(t)=\rho F_3(t), \quad \forall  t>0.
\ee
\end{lemma}
\begin{proof}For any  constants  $l_1,l_2>0$ and $\varepsilon\in(0,1)$, using the fact that $a_2\geq 1$ and similar computations to the ones leading to \eqref{LY-re-comp}, we infer that
\be\label{LY-re-comp2}
\begin{split}
&-\frac{d}{dt}\int_\Omega \left [
       \left(
         n_1-1-\log n_1
       \right)+l_1n_2
       + \frac{l_2}{2}
         c^2\right]\\
&\geq  \mu_1(1-\varepsilon)\int_\Omega  ( n_1-1)^2+[l_2\mu_2-
\frac{(a_1\mu_1+\mu_2l_1)^2}{4\mu_1\varepsilon}]\int_\Omega n_2^2+ (l_2-\frac{\chi_1^2}{4})\int_\Omega |\nabla c|^2.
\end{split}
\ee
Now, thanks to $a_1<1$, we set
$$
l_1= \frac{a_1\mu_1}{\mu_2}, \quad l_2=\frac{\chi_2^2}{4}, \quad \quad \varepsilon=\frac{1+a_1}{2}\in (0, 1),
$$
and then we easily conclude \eqref{decay-lyIII} upon trivial computations from \eqref{LY-re-comp2}.
\end{proof}

\section{Convergence rates}
Aided by those dissipation estimates as provided in Lemmas  \ref{LY1}, \ref{C2*} and \ref{C3*}, even weaker regularity properties than those in Section 2, using the quite known arguments, cf. \cite{B-W, Xinru_et al17,HKMY_1,Tao-Winkler_2015_mu23, TW-2015-SIAMJMA, TW-2016, W-2014} for example,  we know that any  global-in-time and bounded  classical solution of \eqref{P} satisfies the convergence properties as follows:
  \be\label{convergence-n1-n2-c-u}
\Bigr\|(n_1(\cdot, t),n_2(\cdot, t),c(\cdot, t), u(\cdot, t))-(\hat{N}_1, \hat{N}_2, 0, 0)\Bigr\|_{L^\infty(\Omega)}\rightarrow 0 \mbox{ as }\ t \to \infty.
\ee
Here, $(\hat{N}_1, \hat{N}_2)=(N_1,N_2)$ when  $a_1, a_2 \in (0, 1)$,  $(\hat{N}_1, \hat{N}_2)=(0,1)$ when $a_1\geq 1>a_2$, and $(\hat{N}_1, \hat{N}_2)=(1,0)$ when $a_2\geq 1>a_1$, where $N_1$ and $N_2$ are defined by \eqref{N1N2-def}. In this section, we derive  the explicit rates of convergence  as described in Theorem \ref{convergence thm} for any supposedly bounded and global-in-time solution to \eqref{P}. Firstly, it follows from \eqref{convergence-n1-n2-c-u}, as $t\to\infty$, that  $n_1(\cdot,t)\to  \hat{N}_1$ and  $n_2(\cdot,t) \to \hat{N}_2$ uniformly in $\Omega$. Henceforth, we fix a $t_0>1$ such that
\begin{equation}\label{LSB}
  \frac{\hat{N}_1}{2}\leq n_1 \leq \frac{3\hat{N}_1}{2} \quad  \mathrm{ and }\quad  \frac{\hat{N}_2}{2}\leq n_2 \leq \frac{3\hat{N}_2}{2} \quad  \text{on  }  \Omega\times[t_0, \infty).
\end{equation}
\subsection{Convergence rate of $c$.}
We first take up  the convergence rate of $c$, which is based on a parabolic comparison argument.
\begin{lemma}\label{DC}The $c$-solution component of any bounded solution of \eqref{P} stabilizes to zero exponentially:
\begin{equation*}\label{DC-1}
\|c(\cdot,t)\|_{L^\infty}\leq\|c_0\|_{L^\infty}e^{-\frac{\alpha \hat{N}_1+\beta \hat{N}_2}{2}(t-t_0)}, \quad \forall t\geq t_0.
\end{equation*}
\end{lemma}
\begin{proof}
We show the proof only for the case that $a_1, a_2 \in (0, 1)$. Then \eqref{LSB} shows
\begin{equation*}\label{Dc-2}
\alpha n_1 + \beta n_2\geq \frac{\alpha N_1+\beta N_2}{2} \quad  \text{on  }   \Omega\times[t_0, \infty).
\end{equation*}
This along  with the third equation in \eqref{P} and the positive of $c$ gives
\begin{equation}\label{DC-3}
\begin{split}
c_t\leq \Delta c- u\cdot \nabla c-\frac{\alpha N_1+\beta N_2}{2} c \quad  \text{on  }   \Omega\times[t_0, \infty).
\end{split}
\end{equation}
Let $z(t)$ be the solution of the following associated ODE problem:
\begin{equation*}\label{DC-4}
\begin{cases}
z'(t)+\frac{\alpha N_1+\beta N_2}{2} z(t)=0, \ \  t\geq t_0,\\
z(t_0)=\|c(\cdot,t_0)\|_{L^\infty}.
\end{cases}
\end{equation*}
It is clear that $z(t)$ satisfies \eqref{DC-3} together with $\partial_\nu z=0$, and hence an application of the comparison principle and Hopf boundary point lemma  immediately yields
\begin{equation*}\label{DC-5}
c(x,t)\leq z(t)=\|c(\cdot,t_0)\|_{L^\infty}e^{-\frac{\alpha N_1+\beta N_2}{2}(t-t_0)} \ \ \  \mathrm{for~ all}~x\in\Omega, t\geq t_0.
\end{equation*}
Using the basic fact that  $t\to\|c(\cdot,t)\|_{L^\infty}$~is~non-increasing again by comparison principle, cf. \cite[Lemma 2.1]{W-2014}, one has
\begin{equation*}\label{DC-6}
c(x,t)\leq \|c_0\|_{L^\infty}e^{-\frac{\alpha N_1+\beta N_2}{2}(t-t_0)} \ \ \  \mathrm{for~ all}~x\in\Omega, t\geq t_0.
\end{equation*}
This completes the proof of Lemma \ref{DC} by noting the positivity of $c$.
\end{proof}

\subsection{Convergence rates in  Case I: $a_1, a_2 \in (0, 1)$}  In this case, we will show that the solution components $(n_1,n_2,u)$ converge at least exponentially to $(N_1,N_2, 0)$.

 \subsubsection{Convergence rates of $n_1$ and $n_2$ in Case I} In this subsection, we shall establish the convergence rate of $n_1$ and $n_2$ on the basis of the convergence rate of $c$ in Lemma \ref{DC} and the regularity of $n_1$ and $n_2$ provided by Lemma \ref{KS}. We  first  use Lemmas \ref{DC} and \ref{LY1} to obtain the exponential convergence rate of $\|n_1-N_1\|_{L^2}$ and $\|n_2-N_2\|_{L^2}$.
\begin{lemma}\label{L2-u1u2}
The $n_1$- and $n_2$- solution components of bounded solution of \eqref{P} verify
     \begin{equation}\label{L2e}
     \|n_1(\cdot, t) - N_1\|_{L^2}^2+ \|n_2(\cdot, t) - N_2\|_{L^2}^2 \leq K_1 e^{-\kappa (t-t_0)},  \quad \forall  t\geq t_0,
     \end{equation}
     where $K_1=K_1(t_0)$ is defined by
\be\label{K1-def}
K_1=\frac{9\Bigr[E_1(t_0)+\frac{(1-a_1a_2)\mu_1|\Omega|\|c_0\|_{L^\infty}^2\min\Bigr\{\frac{1}{2}, \frac{a_1}{(1+a_1a_2)a_2}\Bigr\} (\frac{N_1\chi_1^2}{4}+\frac{a_2\mu_2N_2\chi_2^2}{4a_1\mu_1}+1)}{(\alpha N_1+\beta N_2)e\max\{\frac{1}{N_1},\frac{a_1\mu_1}{a_2\mu_2N_2}\}}\Bigr]}{2\min\Bigr\{\frac{1}{N_1}, \frac{a_1\mu_1}{a_2\mu_2N_2}\Bigr\}}
\ee
and  the exponential decay rate $\kappa$ is defined by \eqref{kappa-def}.
\end{lemma}
\begin{proof} Applying  Taylor's formula to the function $\psi(z)=z-N_1\ln z$ at $z=N_1$,  we obtain
\begin{equation}\label{LSB-1}
\begin{split}
n_1-N_1-N_1\log\frac{n_1}{N_1}
=\psi(n_1)-\psi(N_1)
=\frac{\psi''(\xi)}{2}(n_1-N_1)^2=\frac{N_1}{2\xi^2}(n_1-N_1)^2
\end{split}
\end{equation}
for some  $\xi>0$ between $n_1$ and $N_1$.  Then the combination of \eqref{LSB} and \eqref{LSB-1} gives
$$
\frac{2}{9N_1}(n_1-N_1)^2\leq n_1-N_1-N_1\log\frac{n_1}{N_1}\leq \frac{2}{N_1}(n_1-N_1)^2, \quad \forall t\geq t_0,
$$
and hence
\begin{equation}\label{LSB-3}
\frac{2}{9N_1}\int_\Omega(n_1-N_1)^2\leq \int_\Omega \left(n_1-N_1-N_1\log\frac{n_1}{N_1}\right)\leq \frac{2}{N_1} \int_\Omega (n_1-N_1)^2, \quad \forall t\geq t_0.
\end{equation}
Similarly, one gets
\begin{equation}\label{LSB-4}
\frac{2}{9N_2}\int_\Omega(n_1-N_2)^2\leq \int_\Omega \left(n_2-N_2-N_2\log\frac{n_2}{N_2}\right)\leq \frac{2}{N_2} \int_\Omega (n_2-N_2)^2, \quad \forall t\geq t_0.
\end{equation}
On the other hand,   Lemma \ref{DC} quickly gives rise to
\begin{equation}\label{LSB-5}
\int_\Omega c^2\leq |\Omega|\|c(\cdot,t)\|_{L^\infty}^2\leq  |\Omega|\|c_0\|_{L^\infty}^2 e^{-(\alpha N_1+\beta N_2)(t-t_0)}, \quad \forall t\geq t_0.
\end{equation}
A substitution of \eqref{LSB-3}, \eqref{LSB-4} and \eqref{LSB-5} into the definition of $E_1$ in \eqref{E1-def}  gives
\be\label{LSB-6}
\begin{split}
E_1(t)&\leq \frac{2}{N_1} \int_\Omega (n_1-N_1)^2+\frac{2a_1\mu_1}{a_2\mu_2N_2}\int_\Omega  (n_2-N_2)^2+m_ce^{-(\alpha N_1+\beta N_2)(t-t_0)}\\[0.25cm]
&\leq \theta F_1+m_c e^{-(\alpha N_1+\beta N_2)(t-t_0)},  \quad \forall t\geq t_0,
\end{split}
\ee
where
\be\label{theta-mc}
\theta=2\max\{\frac{1}{N_1},\frac{a_1\mu_1}{a_2\mu_2N_2}\}, \quad m_c=|\Omega|\|c_0\|_{L^\infty}^2(\frac{N_1\chi_1^2}{4}+\frac{a_2\mu_2N_2\chi_2^2}{4a_1\mu_1}+1).
\ee
 Hence,  from \eqref{LSB-6} and the dissipation estimate \eqref{EF1}, we derive that
\begin{equation*}
\frac{d}{dt} E_1+\frac{\tau}{\theta} E_1\leq \frac{\tau m_c}{\theta}  e^{-(\alpha N_1+\beta N_2)(t-t_0)}, \quad \forall t\geq t_0,
\end{equation*}
and then, solving this Gronwall differential inequality, we  readily get
\be\label{LSB-7}
\begin{split}
E_1(t)&\leq  E_1(t_0)e^{-\frac{\tau}{\theta}(t-t_0)}+\frac{\tau m_c}{\theta} e^{(\alpha N_1+\beta N_2)t_0}e^{-\frac{\tau}{\theta}t}\int_{t_0}^t e^{[\frac{\tau}{\theta}-(\alpha N_1+\beta N_2)]s}ds\\
&\leq \Bigr[E_1(t_0)+\frac{2\tau m_c}{(\alpha N_1+\beta N_2)e\theta}\Bigr]e^{-\min\{\frac{\tau}{\theta}, \frac{(\alpha N_1+\beta N_2)}{2}\}(t-t_0)}, \quad \forall t\geq t_0,
\end{split}
\ee
where we have used the following algebraic calculations:
\be\label{exp-comp} \begin{split}
 &e^{(\alpha N_1+\beta N_2)t_0}e^{-\frac{\tau}{\theta}t}\int_{t_0}^t e^{[\frac{\tau}{\theta}-(\alpha N_1+\beta N_2)]s}ds\\
 &\quad =\begin{cases}
(t-t_0) e^{-\frac{\tau}{\theta}(t-t_0)}, &\text{ if  } \frac{\tau}{\theta}=(\alpha N_1+\beta N_2)\\[0.25cm]
\frac{1}{\frac{\tau}{\theta}-(\alpha N_1+\beta N_2)}\Bigr[e^{-(\alpha N_1+\beta N_2)(t-t_0)}-e^{-\frac{\tau}{\theta}(t-t_0)}\Bigr], & \text{ if }\frac{\tau}{\theta}\neq (\alpha N_1+\beta N_2)
\end{cases}\\[0.25cm]
 &\quad \leq  (t-t_0) e^{-(\alpha N_1+\beta N_2)(t-t_0)}\leq \frac{2}{(\alpha N_1+\beta N_2)e} e^{-\frac{(\alpha N_1+\beta N_2)}{2}(t-t_0)}, \quad \quad \forall t\geq t_0.
\end{split}
\ee
By the definition of $E_1$ in \eqref{E1-def} and the estimates \eqref{LSB-3}, \eqref{LSB-4}, we see that
\be\label{E1-geq-sth}
E_1(t)\geq \frac{2}{9}\min\Bigr\{\frac{1}{N_1}, \frac{a_1\mu_1}{a_2\mu_2N_2}\Bigr\}\Bigr[\int_\Omega (n_1-N_1)^2
     +
     \int_\Omega (n_2-N_2)^2\Bigr].
\ee
Joining \eqref{E1-geq-sth} and \eqref{LSB-7} and substituting the definitions of $\tau$,  $\theta$ and $m_c$   in \eqref{EF1} and \eqref{theta-mc},  we finally arrive at  \eqref{L2e} with $K_1$ and $\kappa$ given by \eqref{K1-def} and \eqref{kappa-def}, respectively.
\end{proof}

Thanks to the regularity in Lemma \ref{KS}, we employ the well-known  Gagliardo-Nirenberg interpolation inequality pass the $L^2$-convergence of $n_1$ and $n_2$ in \eqref{L2e} to the $L^\infty$-convergence.

\begin{lemma}\label{n1n2-conv-rate} Let $\Omega\subset \mathbb{R}^d$  be a bounded and smooth domain. Then the $n_1$- and $n_2$- solution components of bounded solution of \eqref{P} decay exponentially to $(N_1,N_2)$:
\begin{equation}\label{n1-n2-con-rate}
 \|n_1(\cdot, t) - N_1\|_{L^\infty}+ \|n_2(\cdot, t) - N_2\|_{L^\infty} \leq C e^{-\frac{\kappa}{d+2} (t-t_0)},  \quad \forall t\geq t_0
\end{equation}
for some $C>0$ independent of $t$. Here,  the exponential decay rate  $\kappa$ is defined by \eqref{kappa-def}.
\end{lemma}
\begin{proof} Due  to the $L^2$-convergence of $n_1, n_2$ in \eqref{L2e} and the uniform $W^{1,\infty}$-boundedness of $n_1,n_2$ in \eqref{KS-1}, the Gagliardo-Nirenberg inequality  enables us to conclude that
\begin{align*}
&\ \ \|n_1(\cdot,t)-N_1\|_{L^\infty}+\|n_2(\cdot,t)-N_2\|_{L^\infty}\\
&\leq c_1\Bigr(\|n_1(\cdot,t)\|_{W^{1,\infty}}^\frac{d}{d+2}\|n_1(\cdot,t)-N_1\|_{L^2}^\frac{2}{d+2}
+\|n_2(\cdot,t)\|_{W^{1,\infty}}^\frac{d}{d+2}\|n_2(\cdot,t)-N_2\|_{L^2}^\frac{2}{d+2}\Bigr)\\
&\leq c_2\Bigr(\|n_1(\cdot,t)-N_1\|_{L^2}^\frac{2}{d+2}+\|n_2(\cdot,t)-N_2\|_{L^2}^\frac{2}{d+2}\Bigr)\leq c_3 e^{-\frac{\kappa}{d+2} (t-t_0)}, \quad \forall t\geq t_0.
\end{align*}
This is nothing but the exponential decaying estimate \eqref{n1-n2-con-rate}.
\end{proof}

\subsubsection{Convergence rate of $u$ in Case I} With the information gained from Lemmas \ref{L2-u1u2} and \ref{BS}, we are now able to derive  the convergence rate of $u$ in $L^\infty$-norm.  To accomplish this goal, we again begin with its  convergence rate in $L^2$-norm.
\begin{lemma} \label{Lue} The $u$-solution component of \eqref{P} fulfills
\begin{equation}\label{Lue-1}
\|u(\cdot,t)\|_{L^2}^2\leq \Bigr(\|u(\cdot,t_0)\|_{L^2}^2+\frac{2K_1\tilde{K}_1}{\kappa e}\Bigr)e^{- \min\{\lambda_P,\frac{\kappa}{2}\}(t-t_0)}, \quad \forall t\geq t_0,
\end{equation}
where $K_1, \tilde{K}_1$,  $\kappa$ and $\lambda_P$ are respectively defined by \eqref{K1-def}, \eqref{K1-tilde}, \eqref{kappa-def} and \eqref{poincare}.
\end{lemma}
\begin{proof}Recalling that $\nabla \cdot u=0$ in $\Omega$ and $u|_{\partial \Omega}=0$, we  multiply the fourth equation in \eqref{P} by $u$ and integrate it over $\Omega$ to obtain
\begin{equation}\label{Lue-2}
\begin{split}
\frac{1}{2}\frac{d}{dt}\int_\Omega |u|^2+\int_\Omega |\nabla u|^2
&=\int_\Omega (\gamma n_1+\delta n_2)\nabla \phi \cdot u\\
&=\gamma\int_\Omega  (n_1-N_1)\nabla \phi \cdot u+\delta \int_\Omega (n_2-N_2)\nabla \phi \cdot u,
\end{split}
\end{equation}
where we also used the fact  $\int_\Omega \nabla \phi \cdot u =0$. Therefore, we apply the Poincar\'{e} inequality:
\begin{equation}\label{poincare}
\lambda_P\int_\Omega |u|^2\leq \int_\Omega |\nabla u|^2
\end{equation}
for the  Poincar\'{e}  constant $\lambda_P$  to \eqref{Lue-2} deduce that
\begin{equation*}
\begin{split}
&\frac{d}{dt}\int_\Omega |u|^2+2\lambda_P\int_\Omega |u|^2\\
&\leq 2\gamma\int_\Omega  |n_1-N_1| |\nabla \phi \cdot u|+2\delta \int_\Omega |n_2-N_2||\nabla \phi \cdot u|\\
&\leq \lambda_P\int_\Omega |u|^2 +\frac{2\gamma^2\|\nabla \phi\|_{L^\infty}^2}{\lambda_P}\int_\Omega |n_1-N_1|^2 +\frac{2\delta^2\|\nabla \phi\|_{L^\infty}^2}{\lambda_P}\int_\Omega |n_2-N_2|^2.
\end{split}
\end{equation*}
As a result, for
\be\label{K1-tilde}
\tilde{K}_1=2\max\{\frac{\gamma^2}{\lambda_P},\ \  \frac{\delta^2}{\lambda_P}\}\|\nabla \phi\|_{L^\infty}^2,
\ee
it follows that
\begin{equation}\label{Lue-3}
\frac{d}{dt}\int_\Omega |u|^2+\lambda_P\int_\Omega |u|^2\leq  \tilde{K}_1\Bigr(\int_\Omega |n_1-N_1|^2+\int_\Omega |n_2-N_2|^2\Bigr).
\end{equation}
Substituting \eqref{L2e} into \eqref{Lue-3}, we derive that
$$
\frac{d}{dt}\int_\Omega |u|^2+\lambda_P\int_\Omega |u|^2\leq K_1\tilde{K}_1 e^{-\kappa (t-t_0)}, \quad \forall t\geq t_0.
$$
Solving this  ODI and performing similar computations to \eqref{exp-comp},  we readily obtain
\begin{align*}
\|u(\cdot,t)\|_{L^2}^2&\leq  \|u(\cdot,t_0)\|_{L^2}^2e^{-\lambda_P(t-t_0)}+K_1\tilde{K}_1 e^{\kappa t_0}e^{-\lambda_Pt}\int_{t_0}^t e^{(\lambda_P-\kappa)s}ds\\
&\leq \Bigr(\|u(\cdot,t_0)\|_{L^2}^2+\frac{2K_1\tilde{K}_1}{\kappa e}\Bigr)e^{- \min\{\lambda_P,\frac{\kappa}{2}\}(t-t_0)}, \quad \forall t\geq t_0,
\end{align*}
which is precisely the desired exponential decay estimate  \eqref{Lue-1}.
\end{proof}

\begin{lemma}\label{uLinfty I}Let $\Omega\subset \mathbb{R}^d$  be a bounded and smooth domain. Then, for any $\epsilon\in(0,1)$,  there exists a constant $C>0$ such that
\begin{equation}\label{LI-1}
\|u(\cdot,t)\|_{L^\infty}\leq C e^{-\frac{\epsilon}{d+2}\min\{\lambda_P,\frac{\kappa}{2}\}(t-t_0)}, \quad \forall t\geq t_0,
\end{equation}
where the exponent  rate $\kappa$ is defined by \eqref{kappa-def}.
\end{lemma}
\begin{proof} Thanks to the $L^2$-convergence of $u$ in \eqref{Lue-1} and the uniform $W^{1,p}$-boundedness of $u$ in \eqref{Lu1e-1}, the Gagliardo-Nirenberg inequality allows us to infer
\begin{equation*}
\|u(\cdot,t)\|_{L^\infty}\leq c_1\|u(\cdot,t)\|_{W^{1,p}}^\frac{dp}{dp+2p-2d}\|u(\cdot,t)\|_{L^2}^\frac{2p-2d}{dp+2p-2d}\leq c_2e^{- \frac{(p-d)}{dp+2p-2d}\min\{\lambda_P,\frac{\kappa}{2}\}(t-t_0)}, \quad \forall t\geq t_0,
\end{equation*}
which implies \eqref{LI-1} upon choosing $p=[d+2(1-\epsilon)]d/[(d+2)(1-\epsilon)](>d)$.
\end{proof}

\subsection{Convergence rates in  Case II:  $a_1\geq 1>a_2$}  In this section, we will show that the solution components $(n_1,n_2,u)$ converge at least algebraically to $(0,1, 0)$.
\subsubsection{Convergence rates of $n_1$ and $n_2$ in Case II} Again, we start  with the $L^1$- and $L^2$-convergence rates of $n_1$ and $n_2$.
\begin{lemma}There exists $t_1\geq \max\{1, t_0\}$  such that
     \begin{equation}\label{C2-e}
     \|n_1(\cdot, t)\|_{L^1}+ \|n_2(\cdot, t) - 1\|_{L^2}^2 \leq  \frac{K_2}{t+t_1}, \quad \quad \forall  t\geq t_1,
\end{equation}
where
\be\label{K2-def}
K_2=\frac{\max\Bigr\{ 2t_1E_2(t_1), \quad \frac{k_1\Bigr(k_1+\sqrt{k_1^2+2a\sigma}\Bigr)}{\sigma}\Bigr\}}{\min\{1, \frac{2\mu_1}{9a_2\mu_2}\}}\geq O(1)(1+\frac{1}{\sigma})=O(1)(1+(1-a_2)^{-1})
\ee
with $\sigma$, $k_1,k_2$,  $k_3$ and $a$ defined in \eqref{C2*-1}, \eqref{k1-k2-def}, \eqref{E2-up-bd} and \eqref{a-def}, respectively.
\end{lemma}
\begin{proof} Our proof  makes use of the explicit Lyapunov functional provided by Lemma \ref{C2*}.
To proceed,  we first apply the H\"{o}lder inequality to find
\begin{equation}\label{C2-1}
\int_\Omega n_1\leq |\Omega|^\frac{1}{2}\left(\int_\Omega n_1^2\right)^\frac{1}{2}.
\end{equation}
Next, since $\|n_2(\cdot,t)-1\|_{L^\infty}\to 0$ as $t\rightarrow \infty$ and
$$
\lim_{z\rightarrow 1}\frac{z-1-\log z}{z-1}=0,
$$
we can take $\hat{t}_0>0$ such that $|n_2(x,t)-1-\log n_2(x,t)|\leq |n_2(x,t)-1|$ for all $x\in\Omega$ and  $t\geq \hat{t}_0$. Accordingly, we have
\begin{equation}\label{C2-3}
\int_\Omega \left(n_2-1-\log n_2\right)\leq \int_\Omega
|n_2-1|\leq  |\Omega|^\frac{1}{2}\left(\int_\Omega (n_2-1)^2\right)^\frac{1}{2}, \quad \forall t\geq \hat{t}_0.
\end{equation}
In this case,  $(\hat{N}_1, \hat{N}_2)=(0,1)$, so the exponential decay of $c$ in Lemma \ref{DC} warrants that
\begin{equation}\label{C2-4}
\int_\Omega c^2\leq  |\Omega | \|c\|_{L^\infty}^2\leq |\Omega | \|c_0\|_{L^\infty}^2e^{-\beta (t-t_0)}, \quad \forall t\geq t_0.
\end{equation}
From the definitions  of $E_2$ and $F_2$ in Lemma \ref{C2*},  upon a combination of \eqref{C2-1}, \eqref{C2-3} and \eqref{C2-4} and the fact that $\sqrt{A}+\sqrt{B}\leq \sqrt{2(A+B)}$ for $A, B\geq 0$, we find two constants
\be\label{k1-k2-def}
k_1=\max\{1,\frac{\mu_1}{a_2\mu_2}\}(2|\Omega|)^\frac{1}{2},\quad \quad \quad  k_2=\frac{\mu_1\chi_2^2}{8a_2\mu_2}|\Omega | \|c_0\|_{L^\infty}^2
\ee
such that
\begin{equation*}
E_2(t)\leq k_1 F_2^\frac{1}{2}(t)+ k_2e^{-\beta(t-t_0)},\quad \forall  t\geq \max\{t_0,\hat{t}_0\},
\end{equation*}
which further gives us
$$
E_2^2(t)\leq 2k_1^2 F_2(t)+2k_2^2 e^{-2\beta (t-t_0)}, \quad \forall  t\geq \max\{t_0,\hat{t}_0\}.
$$
A substitution of this into the dissipation inequality  \eqref{C2*-1} entails
\begin{equation}\label{C2-5}
\frac{d}{dt}E_2(t)+\frac{\sigma}{2k_1^2} E^2_2(t)\leq \frac{k_2^2\sigma}{k_1^2} e^{-2\beta(t-t_0)}, \quad  t\geq \max\{t_0,\hat{t}_0\}.
\end{equation}
Now, to illustrate \eqref{C2-e}, we first take $t_1=\max\Bigr\{t_0,\hat{t}_0, 1, \beta^{-1}\Bigr\}$ so that
\be\label{a-def}
a=\frac{k_2^2\sigma}{k_1^2} \max\Bigr\{(t+t_1)^2e^{-2\beta(t-t_0)}:   t\geq t_0\Big\}=\frac{4k_2^2t_1^2\sigma}{k_1^2} e^{-2\beta(t_1-t_0)};
\ee
and  then, for any
\be\label{b-def}
b\geq \frac{k_1\Bigr(k_1+\sqrt{k_1^2+2a\sigma}\Bigr)}{\sigma},
\ee
we put
\be\label{upp-sol}
y(t)=\frac{b}{t+t_1},\quad \quad  t\geq 0.
\ee
We use  straightforward calculations from \eqref{upp-sol} and use \eqref{a-def} to see that
\begin{align*}
& y^\prime (t)+\frac{\sigma}{2k_1^2} y^2(t)-\frac{k_2^2\sigma}{k_1^2} e^{-2\beta(t-t_0)}\\
&=(t+t_1)^{-2}\Bigr[\frac{\sigma}{2k_1^2}b^2-b-\frac{k_2^2\sigma}{k_1^2} (t+t_1)^2e^{-2\beta(t-t_0)}\Bigr]\\
&\geq (t+t_1)^{-2}(\frac{\sigma}{2k_1^2}b^2-b-a)\geq 0, \quad \forall t\geq t_1.
\end{align*}
This, upon a clear choice of $b$ in \eqref{b-def},  an ODE comparison argument to \eqref{C2-5} shows
\begin{equation}\label{E2-up-bd}
E_2(t)\leq \frac{\max\Bigr\{ 2t_1E_2(t_1), \quad \frac{k_1\Bigr(k_1+\sqrt{k_1^2+2a\sigma}\Bigr)}{\sigma}\Bigr\}}{t+t_1}:=\frac{k_3}{t+t_1}, \quad \forall t\geq t_1.
\end{equation}
Then since $t_1\geq t_0$, we infer from  \eqref{LSB-4}  and the definition of $E_2(t)$  in Lemma \ref{C2*} that
\begin{equation*}
\min\{1, \frac{2\mu_1}{9a_2\mu_2}\}\Bigr(\|n_1(\cdot,t)\|_{L^1}+\|n_2(\cdot,t)-1\|^2_{L^2}\Bigr)\leq  \frac{k_3}{t+t_1}, \quad \quad \forall t\geq t_1.
\end{equation*}
This, upon a substitution of the respective definitions of  $k_1,k_2$,  $k_3$ and $a$ in \eqref{k1-k2-def}, \eqref{E2-up-bd} and \eqref{a-def},  proves our desired algebraic decay estimate  \eqref{C2-e}.
\end{proof}

\begin{lemma}\label{n1n2-conv-rate II} Let $\Omega\subset \mathbb{R}^d$  be a bounded and smooth domain. Then the $n_1$-and $n_2$-solution components of bounded solution of \eqref{P} decay at least algebraically to $(0,1)$: there exist two constants $K_3$ and $K_4$ independent of $t$ fulfilling
$$
K_3\geq O(1)\Bigr(1+(1-a_2)^{-\frac{1}{d+1}}\Bigr), \quad \quad K_4\geq O(1)\Bigr(1+(1-a_2)^{-\frac{1}{d+2}}\Bigr)
$$
such that
\be\label{n1-n2-con-rate IIn1}
 \|n_1(\cdot, t)\|_{L^\infty} \leq  \frac{K_3}{(t+t_1)^{\frac{1}{d+1}}}, \quad \forall  t\geq t_1
 \ee
 as well as
 \be \label{n1-n2-con-rate IIn2}
 \|n_2(\cdot, t) - 1\|_{L^\infty} \leq \frac{K_4}{(t+t_1)^{\frac{1}{d+2}}},  \quad \forall t\geq t_1.
\ee
\end{lemma}
\begin{proof}
Equipped with the uniform $W^{1,\infty}$-bounds of $n_1, n_2$ in Lemma  \ref{KS}, as before, by means of the Gagliardo-Nirenberg inequality, we readily infer, for all $t\geq t_1$,
\begin{align*}
 \|n_1(\cdot,t)\|_{L^\infty} \leq c_1\Bigr\|n_1(\cdot,t)\|_{W^{1,\infty}}^\frac{d}{d+1}\|n_1(\cdot,t)\|_{L^1}^\frac{1}{d+1}\leq c_2 \|n_1(\cdot,t)\|_{L^1}^\frac{1}{d+1}
\end{align*}
and
\begin{align*}
\|n_2(\cdot,t)-1\|_{L^\infty}\leq c_3\|n_2(\cdot,t)\|_{W^{1,\infty}}^\frac{d}{d+2}\|n_2(\cdot,t)-1\|_{L^2}^\frac{2}{d+2}\leq c_4\|n_2(\cdot,t)-1\|_{L^2}^\frac{2}{d+2}.
\end{align*}
These along with  the $L^1$-and $L^2$-convergence of $n_1, n_2$ in \eqref{C2-e} and the bound for $K_2$ in \eqref{K2-def} yield  immediately \eqref{n1-n2-con-rate IIn1} and  \eqref{n1-n2-con-rate IIn2}.
\end{proof}
\subsubsection{Convergence rate of $u$ in Case II}
\begin{lemma} \label{u-con case II} The $u$-solution component of \eqref{P} fulfills
\begin{equation}\label{Lue-1-II}
\|u(\cdot,t)\|_{L^2}^2\leq \frac{K_5}{t+t_1}, \ \ \quad \forall t\geq t_1
\end{equation}
for some positive constant $K_5$ independent of $t$ satisfying $
K_5\geq O(1)(1+(1-a_2)^{-1})$.
\end{lemma}
\begin{proof} Thanks to the $L^1$-and $L^2$-convergence of $n_1, n_2$ in \eqref{C2-e}, we can easily adapt the proof of Lemma \ref{Lue} here. Indeed, we derive from \eqref{Lue-2},  the Poincar\'{e} inequality \eqref{poincare} and the boundedness of $u$ that
\be\label{dif-II}
\frac{1}{2}\frac{d}{dt}\int_\Omega |u|^2+\frac{\lambda_P}{2}\int_\Omega |u|^2\leq \gamma\|\nabla \phi\|_{L^\infty}\|u\|_{L^\infty}\int_\Omega n_1+ \frac{\delta^2\|\nabla \phi\|_{L^\infty}^2}{2\lambda_P}\int_\Omega |n_2-1|^2.
\ee
Thus, for
$$
\tilde{K}_5=2\max\{ \gamma\|u\|_{L^\infty(\Omega\times (0,\infty))},\ \  \frac{\delta^2\|\nabla \phi\|_{L^\infty}}{2\lambda_P}\}\|\nabla \phi\|_{L^\infty}<\infty,
$$
from \eqref{dif-II} and  \eqref{C2-e},  we obtain  an ODI as follows:
\begin{equation}\label{Lue-3II}
\frac{d}{dt}\int_\Omega |u|^2+\lambda_P\int_\Omega |u|^2\leq  \tilde{K}_5\Bigr( \int_\Omega n_1+\int_\Omega|n_2-1|^2\Bigr)\leq \frac{K_2\tilde{K}_5}{t+t_1}, \quad \forall t\geq t_1.
\end{equation}
Solving the  ODI \eqref{Lue-3II}, we end up with
\begin{align*}
\|u(\cdot,t)\|_{L^2}^2&\leq  \|u(\cdot,t_1)\|_{L^2}^2e^{-\lambda_P(t-t_1)}+K_2\tilde{K}_5 e^{-\lambda_Pt}\int_{t_1}^t \frac{e^{\lambda_P s}}{s+t_1}ds\\
&\leq \|u(\cdot,t_1)\|_{L^2}^2e^{-\lambda_P(t-t_1)}+\frac{K_2\tilde{K}_5\hat{K}_5}{t+t_1}\\
&\leq \Bigr(\|u(\cdot,t_1)\|_{L^2}^2e^{-\lambda_P t_1}\bar{K}_5+K_2\tilde{K}_5\hat{K}_5\Bigr)(t+t_1)^{-1}, \quad \forall t\geq t_1,
\end{align*}
from which \eqref{Lue-1-II} follows. Here,  we have used the following facts
$$
\bar{K}_5=\max\{(t+t_1)e^{-\lambda_Pt}:\ \  t\geq t_1\}<\infty
$$
and
$$
\hat{K}_5=\max\Bigr\{(t+t_1)e^{-\lambda_Pt}\int_{t_1}^t \frac{e^{\lambda_P s}}{s+t_1}ds: \ \ t\geq t_1\Bigr\}<\infty.
$$
The latter is due to
$$
\lim_{t\to \infty} \Bigr[(t+t_1)e^{-\lambda_Pt}\int_{t_1}^t \frac{e^{\lambda_P s}}{s+t_1}ds\Bigr]=\lim_{t\to \infty} \frac{\int_{t_1}^t \frac{e^{\lambda_P s}}{s+t_1}ds}{(t+t_1)^{-1}e^{\lambda_Pt}}=\frac{1}{\lambda_p}<\infty. $$
\end{proof}
With the $L^2$-convergence of $u$ in Lemma \ref{u-con case II} at hand,   the same argument as done to Lemma \ref{uLinfty I} shows the following $L^\infty$-convergence of $u$.
\begin{lemma}\label{uLinfty II}Let $\Omega\subset \mathbb{R}^d$  be a bounded and smooth domain. Then, for any $\epsilon\in(0,1)$,  there exists a positive constant
$$
K_6\geq O(1)\Bigr(1+(1-a_2)^{-\frac{\epsilon}{d+2}}\Bigr)
$$
 such that
\begin{equation}\label{LI-1}
\|u(\cdot,t)\|_{L^\infty}\leq  \frac{K_6}{(t+t_1)^\frac{\epsilon}{d+2}}, \quad \quad \forall t\geq t_1.
\end{equation}
\end{lemma}

\subsection{Convergence rates in  Case III:  $a_2\geq 1>a_1$}  In this case, we  shall show that the solution components $(n_1,n_2,u)$ converge at least algebraically to $(1,0, 0)$. Comparing Lemmas \ref{C2*} and \ref{C3*} and the $n_1$-and $n_2$-equations in \eqref{P}, we see that this subsection is fully parallel to Section 4.3, and so we simply write down their respective  final outcomes.

\begin{lemma}\label{n1n2-conv-rate III}Let $\Omega\subset \mathbb{R}^d$  be a bounded and smooth domain. Then the $n_1$- and $n_2$- solution components of bounded solution of \eqref{P} decay at least algebraically to $(1,0)$: there exist $t_2\geq\max\{1, t_0\}$ and positive constants $K_7$ and $K_8$ independent of $t$ fulfilling
$$
K_7\geq O(1)\Bigr(1+(1-a_1)^{-\frac{1}{d+2}}\Bigr), \quad \quad K_8\geq O(1)\Bigr(1+(1-a_1)^{-\frac{1}{d+1}}\Bigr)
$$
such that
$$
 \|n_1(\cdot, t)-1\|_{L^\infty} \leq  \frac{K_7}{(t+t_2)^{\frac{1}{d+2}}}, \quad \forall  t\geq t_2
$$
and
$$
 \|n_2(\cdot, t) \|_{L^\infty} \leq \frac{K_8}{(t+t_2)^{\frac{1}{d+1}}},  \quad \forall t\geq t_2.
$$
\end{lemma}

\begin{lemma}\label{uLinfty III}Let $\Omega\subset \mathbb{R}^d$  be a bounded and smooth domain. Then, for any $\epsilon\in(0,1)$,  there exists a positive constant $ C_\epsilon \geq O(1)(1+(1-a_1)^{-\frac{\epsilon}{d+2}})$ such that
\begin{equation}\label{LI-1}
\|u(\cdot,t)\|_{L^\infty}\leq  \frac{C_\epsilon}{(t+t_2)^\frac{\epsilon}{d+2}}, \quad \quad \forall t\geq t_2.
\end{equation}
\end{lemma}
\begin{proof}[Proof of Theorem \ref{convergence thm}]
 Notice that $t_0,t_1, t_2\geq 1$; the respective decay estimates asserted in Theorem \ref{convergence thm} follow from  some lemmas in this section  with perhaps some large constants $m_i$. More specifically, the exponential decay estimate  \eqref{exp-convergence-n1-n2-u} follows from Lemmas \ref{n1n2-conv-rate} and \ref{uLinfty I}; the algebraical decay estimate \eqref{alg-convergence-n1-n2-u II} follows from Lemmas \ref{n1n2-conv-rate II} and \ref{uLinfty II}; the algebraical decay estimate \eqref{alg-convergence-n1-n2-u III} follows from Lemmas \ref{n1n2-conv-rate III} and \ref{uLinfty III}; and, finally, the exponential decay estimate \eqref{c-exp-rate} follows simply  from Lemma \ref{DC}.
\end{proof}

\textbf{Acknowledgments}   The research of H.Y. Jin was supported by  NSF of China (No. 11501218) and the Fundamental Research Funds for the Central Universities (No. 2017MS107), and the research of  T. Xiang  was  supported by the NSF of China (No. 11601516).

\end{document}